\documentclass[12pt,reqno] {amsart}
\usepackage{amsthm, mathtools, amsmath,amscd,graphicx,amssymb,
amstext, epsfig,verbatim, multicol,mathrsfs,mathbbol, xcolor}

\newtheorem{thm}{Theorem}

\newtheorem{lem}{Lemma}
\newtheorem{prop}{Proposition}
{\theoremstyle{remark}
\newtheorem {Rem}{Remark}}
{\theoremstyle{definition}
\newtheorem{example}{Example}}
                         \newtheorem{defn}[thm]{Definition}
\def\div{\mathrm{div}}

\def\fchar{\mathrm{char}}

                                                   \newcommand{\CC}{\mathcal C}

\newcommand{\vf}{\varphi}

\newcommand{\eps}{\varepsilon}

\newcommand{\beq}{\begin{equation}}
\newcommand{\eeq}{\end{equation}}

\begin{document}
\begin{abstract}
Let $d>1$ be an integer, $K_0$  a perfect field such that $\fchar(K_0) \nmid d$,  $n>d$ an integer that is prime to $d$.
Let  $f(x)\in K_0[x]$ be a degree $n$ monic polynomial without repeated roots,  and $\mathcal{C}_{f,d}$ a smooth projective  model of  the affine curve $y^d=f(x)$. Let $J(\mathcal{C}_{f,d})$ be the Jacobian of the $K_0$-curve $\mathcal{C}_{f,d} $. As usual, we identify $\mathcal{C}_{f,d}$ with its canonical image in $J(\mathcal{C}_{f,d})$ (such that the only  ``infinite point'' of $\mathcal{C}_{f,d}$ goes to the zero of the group law on $J(\mathcal{C}_{f,d})$).

We say that an integer $m>1$ is $(n,d)$-reachable over $K_0$ if there  exists a polynomial $f(x)$ as above such that $\mathcal{C}_{f,d}(K_0)$ contains a torsion point of order $m$.

Let us put $\ell_0:=[(n+d)/d], \ m_0:=\ell_0 d$.
Earlier we proved  that if $m$ is  $(n,d)$-reachable, then either $m=d$ or $m = n$ or $m \ge m_0$
(in addition, both $d$ and $n$ are  $(n,d)$-reachable over every $K_0$).
We also proved that if  $m_0$ is  $(n,d)$-reachable over some $K_0$ then $n-m_0+\ell_0\ge 0$.

In the present paper we discuss the $(n,d)$-reachability of $m_0$ when
$n-m_0+\ell_0=0$ or $1$.
\end{abstract}

\title[Torsion Points of small order]{Torsion points of small order on cyclic covers of $\mathbb P^1$. III}
\author {Boris M. Bekker}
 \address{ St.Petersburg State University, 7/9 Universitetskaya nab., St. Petersburg, 199034 Russia.}
\email{ bekker.boris@gmail.com}
\author {Yuri G. Zarhin}
\address{Pennsylvania State University, Department of Mathematics, University Park, PA 16802, USA}
\email{zarhin@math.psu.edu}
\thanks{The second named author (Y.Z.) was partially supported by Simons Foundation Collaboration grant   \# 585711  and the Travel Support for Mathematicians Grant MPS-TSM-00007756 from the Simons Foundation. Most of this work was done in August-September 2025 during his stay at the Max-Planck-Institut f\"ur Mathematik (Bonn, Germany), whose hospitality and support are gratefully acknowledged.}

\maketitle
\section{Introduction}

In what follows, $K$ is an algebraically closed field, $n$ and $d$ are positive integers such that
$$2 \le d<n, \quad (n,d)=1,  \quad \fchar(K) \nmid d.$$
Let $f(x) \in K[x]$ be a degree $n$ polynomial without repeated roots, $\mathcal{C}_{f,d}$  the smooth
projective model of the smooth plane affine curve $C_{f,d}:y^d=f(x)$,  and $\mathcal O=\mathcal O_{f,d}$   the only point of
$\mathcal{C}_{f,d}$ that lies above the {\sl infinite point} of the projective closure of $C_{f,d}$ (see \cite{Arul,BZML}). We may view $x$ and $y$ as rational functions on  $\mathcal{C}_{f,d}$ with only pole at $\mathcal O$;
they generate the field $K(\mathcal{C}_{f,d})$  of rational functions on  $\mathcal{C}_{f,d}$ over $K$.

Further we identify
$\mathcal{C}_{f,d}$ with its image in its jacobian $J(\mathcal{C}_{f,d})$ under the canonical embedding, which sends
$\mathcal O$ to the zero of the group law on $\mathcal{C}_{f,d}$.

We  write $\mu_d$ for the group  of $d$-th roots of unity in $K$, which acts on $\mathcal{C}_{f,d}$.
It was proven in \cite{BZR} that a $K$-point $\mathcal P$  on  $\mathcal{C}_{f,d}$
is a torsion point of order $d$ if and only if $y(\mathcal P)=0$, i.e., $x(\mathcal P)\in K$ is a root of $f(x)$. On the other hand, if $m\neq 1,d$, then
the number of $m$-order points $\mathcal P=(x(\mathcal P),y(\mathcal P)) \in \mathcal{C}_{f,d}(K)$ such that $y(\mathcal P)\neq0$ is divisible by $d$.

We  write $\mu_d$ for the group  of $d$-th roots of unity in $K$. If $m> 1$ is an integer, then by an $ m$-packet on $\mathcal C _{f ,d}$ we mean a nontrivial $\mu_d$ -orbit
consisting of $m$-torsion $K$-points of $\mathcal C_{f ,d}$. (Each such packet consists of $d$ distinct
points.)

Let $K_0$ be a perfect subfield of $K$ such that $f(x)\in K_0[x]$. Then one may view $\mathcal{C}_{f,d}$  as a $K_0$-curve;
in addition, $\mathcal O_{f,d}$ becomes a $K_0$-point of $\mathcal{C}_{f,d}$. One may also view $x$ and $y$ as rational functions that are defined over $K_0$.

 The following definition was introduced in \cite{BZML}.

\begin{defn}
An integer $m>1$ is $(n,d)$-reachable over $K_0$ if there exists a degree $n$ polynomial without repeated roots such that
$\mathcal{C}_{f,d}(K_0)$ contains a torsion point of order $m$.
\end{defn}

In \cite{BZR,BZML} we proved the following assertions.

\begin{itemize}
\item[(i)]
$d$ and $n$ are $(n,d)$-reachable over every $K_0$.
\item[(ii)]
If  $1<m<d$ or $d<m<n$, then $m$ is not  $(n,d)$-reachable over any $K_0$.
\item[(iii)]
If $n<m<2n$ and $m$ is $(n,d)$-reachable over some $K_0$, then either
$d |m$ or $m \equiv n \bmod d$.
\item[(iv)] Let us put
$m_0:=d \cdot [(n+d)/d]$.
If $m>n$ and $m$ is $(n,d)$-reachable over some $K_0$, then either
$m=m_0$ or $m>m_0$.
\item[(v)]
If  $m_0$ is $(n,d)$-reachable over some  $K_0$,  then
$n-m_0+\ell_0 \ge 0$, where
$$\ell_0:= [(n+d)/d]=m_0/d.$$
\item[(vi)]
If $K_0$ is infinite and $n-m_0+\ell_0 \ge 0$, then  $m_0$ is $(n,d)$-reachable over  $K_0$ if $\fchar(K)$
enjoys one of the following properties.
\begin{enumerate}
\item[(1)]
$\fchar(K)=0$.
\item[(2)]
$\fchar(K)>n$.
\item[(3)]
If $n-m_0+\ell_0 =0$, then $\fchar(K)$ does not divide $\ell_0$.
\item[(4)]
If $n-m_0+\ell_0 >0$, then $\fchar(K)$ does not divide $n-m_0+\ell_0$.
\end{enumerate}
\end{itemize}

In this paper we discuss the following two questions. In Section \ref{rational} we construct
versal families of curves of curves  $\mathcal{C}_{f,d}$ with torsion $K_0$-points of order
$m_0$. In addition, we discuss in detail the cases when  $n-m_0+\ell_0=0$  or $1$.
In Section \ref{ellipticCurves} we discuss elliptic curves with torsion points of order 4
(the case of $n=3,d=2, \ell_0=2$).
In Section \ref{2 packets} we deal with curves $\mathcal{C}_{f,d}$ over $K$
that have at least $2d$ points of order $m_0$, paying special attention to the case of hyperelliptic curves
(when $d=2$) in Section \ref{d=2}.

\section{Rationality Questions}
\label{rational}

\vskip.5cm

We may view $x$ and $y$ as rational functions on  $\mathcal{C}_{f,d}$ with only pole at $\mathcal O$;
they generate the field $K(\mathcal{C}_{f,d})$  of rational functions on  $\mathcal{C}_{f,d}$ over $K$.
One may view  $K(\mathcal{C}_{f,d})$ as the $d$-dimensional vector space over the field $K(x)$ of rational functions
with a basis $\{y^i \mid 0 \le i\le d-1\}$, i.e.,
\begin{equation}
\label{Kxy}
 K(\mathcal{C}_{f,d})=\oplus_{i=0}^{d-1} y^i K(x).
 \end{equation}
Let $K_0$ be a perfect subfield of $K$ such that
$$f(x)\in K_0[x].$$
Then one may view  $\mathcal{C}_{f,d}$ as the smooth projective curve over $K_0$ while
$x,y$ lie in the field $K_0(\mathcal{C}_{f,d})$ of rational functions on  $\mathcal{C}_{f,d}$ over $K_0$;
in addition, $\mathcal O$ becomes a $K_0$-point of  $\mathcal{C}_{f,d}$ and
\begin{equation}
\label{K0xy}
 K_0(\mathcal{C}_{f,d})=\oplus_{i=0}^{d-1} y^i K_0(x).
 \end{equation}

 Let  us define the  integer
$$\ell_0:=[(n+d)/d] \ge 2.$$
Then
 $$m_0=d \cdot [(n+d)/d]=d\ell_0$$
is the only integer that lies (strictly) between $n$ and $n+d$ and is divisible by $d$.

We know \cite[Th. 3 and Prop. 2 and 3]{BZML}
that if a curve $\CC_{f,d}$ has a point  $\mathcal P$ of order $m$, where $n<m<n+d$, with abscissa $x(\mathcal P)=a$, then
$$m=m_0,  \quad n-m_0+\ell_0 \ge 0,$$
and the polynomial $f(x)$ can be represented in the form
$$f(x)=A(x-a)^{m_0}+v(x)^d, \ \ \text{ where } \ \ A \in K^{*}, \ \ v(x)\in K[x], \ \ \deg(v)=\ell_0.$$
In addition,
\begin{itemize}
\item[(a)]
 the principal divisor $m(\mathcal P)-m(\mathcal O)$ coincides with $\div(v(x)-y)$;
 \item[(b)]
 if $B$ is the leading coefficient of $v(x)$, then
 $A=-B^{d}$;
 there is a nonzero degree $n-m_0+\ell_0$ polynomial $q(x)\in K[x]$ such that
 $v(x)=B(x-a)^{\ell_0}+q(x)$.
 \end{itemize}

 \begin{prop}
 \label{orderM0}
 Suppose that
 $$n-m_0+\ell_0 \ge 0.$$
 Let $K_0$ be a perfect subfield of $K$,  and  $f(x) \in K_0[x]$ a degree $n$ polynomial without repeated roots.
 \begin{enumerate}
 \item[(1)]
Then the following conditions are equivalent.
 \begin{itemize}
 \item[(i)]
 The curve $\CC_{f,d}$ has a  $K_0$-point  $\mathcal P$ of order $m_0$ with abscissa $x(\mathcal P)=a\in K_0$.
 \item[(ii)]
  There exist a nonzero $B\in K_0$ and a degree $\ell_0$ polynomial $v(x)\in K_0[x]$ with nonzero leading coefficient $B$
such that
$$v(a)\ne 0, \; f(x)=-B^d (x-a)^{m_0}+v(x)^d, \textrm\;{and}\; \mathcal P=(a,v(a)).$$
In addition, there is a nonzero degree $n-m_0+\ell_0$ polynomial $q(x)\in K_0[x]$
such that $q(a)\ne 0$ and
 $v(x)=B(x-a)^{\ell_0}+q(x)$.
 \end{itemize}
 \item[(2)]
 Suppose that $f(x)$ enjoys the equivalent properties {\rm({i})}  and {\rm({ii})}
 and consider the  polynomials
 $$h(x):=v(a)^{-d}f(x+a), \quad w(x):=v(a)^{-1} v(x+a),
 \quad r(x):=v(a)^{-1} q(x+a)$$
 with coefficients in $K_0$,
 Then $h(x), w(x), r(x)$ enjoy the following properties.
 \begin{itemize}
 \item[(a)]
 $h(x)$  is a degree $n$ polynomial without repeated roots and
 $$h(x)=-\tilde{B}^d x^{m_0}+w(x)^d,  \quad
 w(x)=\tilde{B}x^{\ell_0}+r(x)$$
 where
 $$\tilde{B}:=\frac{B}{v(a)} \in K_0.$$
  \item[(b)]
 $$\deg(w)=\ell_0, \quad \deg(r)=n-m_0+\ell_0<\ell_0=\deg(w), \quad w(0) =1.$$
 In addition, $r(x) \not\equiv 0$.
 \item[(c)]
 The map
 $$(x,y) \to (x_1,y_1)=\left(x-a, \frac{y}{v(a)}\right)$$
 defines the isomorphism of plane affine algebraic $K_0$-curves
 $y^d=f(x)$ and $y_1^d=h(x_1)$ that extends to
 the isomorphism of $K_0$-curves
 $\CC_{f,d} \cong \CC_{h,d}$, which sends $\mathcal O_{f,d}$ to $\mathcal O_{h,d}$ and $\mathcal  P$ to $\mathcal Q$.
 \item[(d)]
 $\mathcal Q=(0,1)$ is a $K_0$-point of order $m_0$ on $\CC_{h,d}$.
 \end{itemize}
 \end{enumerate}
 \end{prop}


 \begin{proof}[Proof of Proposition \ref{orderM0}]
 Let $\mathcal P=(a,c)$ be a $K_0$-point of order $m_0$ on $\CC_{f,d}$.  Since $m_0 \ne d$, we have
$c \ne 0.$
 We know that there exists a polynomial $v(x) \in K[x]$ with (nonzero) leading coefficient $B \in K$ such that
 $m_0(\mathcal P)-m_0(\mathcal O)=\div(v(x)-y)$,
 $$f(x)=-B^d (x-a)^{m_0}+v(x)^d,$$
 and
 $$q(x):=v(x)-B(x-a)^{\ell_0)}\in K[x]$$
  is a nonzero   polynomial of degree $n-m_0+\ell_0$.

 Clearly, the principal divisor
 $$\mathcal{D}:=m_0(\mathcal P)-m_0(\mathcal O)=m_0\left((\mathcal P)-(\mathcal  O)\right)$$
  is defined over $K_0$. This implies that there is a nonzero  rational function $h(x)\in K_0(\CC_{f,d})$
 and a nonzero constant $\lambda \in K$ such that
 $$h=\lambda (v(x) -y).$$
 Since $\mathcal P=(a,c)$ is the only zero of $h$, we get $c=v(a)$, hence, $v(a) \ne 0$.
 Since $y \in K_0(x)$,  we conclude that $\lambda \in K_0$, which implies
 (in light of \eqref{K0xy})
 that $v(x) \in K_0[x]$, and therefore its leading coefficient
 $B$ lies in $K_0$. This implies that $q(x)\in K_0[x]$.
 It remains to notice that if $q(a)=0$,
 then $a$ is also a  root of $v(x)$, which implies that $a$ is a multiple root of $f(x)$,
 which is not the case.

 Conversely, suppose that we are given
 $$a, B \in K_0; \quad  q(x), v(x), f(x) \in K_0[x]$$
 as above. Then $\mathcal P=(a,v(a))$ is a $K$-point of order $m_0$ on $\CC_{f,d}$.
 Since $a,v(a)\in K_0$, we conclude that $\mathcal P$  is a $K_0$-point of order $m_0$ on $\CC_{f,d}$.
 This ends the proof of (1).

 The assertion (2) immediately follows from (1) via direct computations.
 \end{proof}

\begin{thm}
\label{m0Norm}
 Suppose that
 $$n-m_0+\ell_0 \ge 0.$$
 Let $K_0$ be a perfect subfield of $K$,  and  $f(x) \in K_0[x]$ a degree $n$ polynomial without repeated roots.
 Then the curve $\CC_{f,d}$ has a  $K_0$-point  $\mathcal P$ of order $m_0$  if and only if there exists a degree $n$ polynomial $\tilde{f}(x)\in K_0[x]$ without repeated roots that enjoys the following properties.

 \begin{itemize}
 \item[(i)]
  There exist a nonzero $\tilde{B}\in K_0$ and a degree $\ell_0$ polynomial $w(x)\in K_0[x]$ with nonzero leading coefficient $\tilde{B}$
such that
$$w(0)=1, \quad \tilde{f}(x)=-\tilde{B}^d x^{m_0}+w(x)^d.$$
\item[(ii)]
There is a nonzero degree $n-m_0+\ell_0$ polynomial $\tilde{q}(x)\in K_0[x]$
such that $\tilde q(0)=1$ and
 $w(x)=\tilde B x^{\ell_0}+\tilde q(x)$.
 \item[(iii)]
 There is an isomorphism of algebraic $K_0$-curves $\CC_{f,d}$ and $\CC_{\tilde{f},d}$
 under which $\mathcal O_{f,d}$ goes to $\mathcal O_{\tilde{f},d}$ and $\mathcal P\in \CC_{f,d}(K_0)$ goes to $\mathcal Q=(0,1)\in \CC_{\tilde{f},d}(K_0)$.
 \end{itemize}
 If this is the case, then $\mathcal Q=(0,1)$ is a $K_0$-point of order $m_0$ on $\CC_{\tilde{f},d}$.

\end{thm}
\begin{proof}
The assertion follows readily from Proposition \ref{orderM0}.

\end{proof}

Now let us discuss the case of small nonnegative $n-m_0+\ell_0$.

\begin{thm}
\label{equal0}
Suppose that $n-m_0+\ell_0=0$. Then the following statements hold.
\begin{itemize}
\item[(i)]
$m_0$ is $(n,d)$-reachable over $K_0$ if and only if $\fchar(K_0)$ does not divide $\ell_0$.
\item[(ii)]
Suppose that $\fchar(K_0)$ does not divide $\ell_0$. Let us consider the polynomial
$$f_0(x)=-x^{m_0}+(x^{\ell_0}+1)^d \in K_0[x].$$
Then $f_0(x)$ is a degree $n$ polynomial without repeated roots and
$\mathcal Q_0=(0,1)$ is a $K_0$-point of order $m_0$ on $\CC_{f_0,d}$.
\item[(iii)]
Suppose that $\fchar(K_0)$ does not divide $\ell_0$ and let $f(x) \in K_0[x]$
be a degree $n$ polynomial without repeated roots such that $\CC_{f,d}$ has a $K_0$-point $\mathcal P$ order $m_0$.
Then there is an isomorphism of curves  $\CC_{f,d}$ and $\CC_{f_0,d}$ over $K$
under which $\mathcal O_{f,d}$ goes to $\mathcal O_{f_0,d}$ and $\mathcal P$ goes to $\mathcal Q_0$.
\end{itemize}
\end{thm}

\begin{proof}
Suppose that $f(x)\in K_0[x]$ is a degree $n$ polynomial without repeated roots
such that $\CC_{f,d}(K_0)$ has a $K_0$-point of order $m_0$. It follows from
Proposition \ref{orderM0} that there exist a nonzero $B \in K_0$,
$a \in K_0$,  a nonzero constant $q(x) =C \in K$ (recall that  $\deg(q)=n-m_0+\ell_0=0$) such that
$$\begin{aligned}f(x)=-B^d (x-a)^{m_0}+\left(B (x-a)^{\ell_0}+C\right)^d\\=-B^d \left((x-a)^d\right)^{\ell_0}+\left(B (x-a)^{\ell_0}+C\right)^d.\end{aligned}$$
If $\fchar(K_0)$ divides $\ell_0$, then the derivative of $f(x)$ is identically $0$, which contradicts our assumption that $f(x)$ has no
repeated roots.  Hence, if $m_0$ is $(n,d)$-reachable over $K_0$, then  $\fchar(K_0)$ does not divide $\ell_0$.

Now let us assume that  $\fchar(K_0)$ does {\sl not} divide $\ell_0$.  Let us consider the polynomial
$$f_0(x)=-x^{m_0}+\left(x^{\ell_0}+1\right)^d=-\left(x^{\ell_0}\right)^d+\left(x^{\ell_0}+1\right)^d \in K_0[x]\subset K[x].$$
Let us check that $f_0(x)$ has no repeated roots. Indeed,  since $\fchar(K)=\fchar(K_0)$ does not divide $d$,
the polynomial
$F_0(X):=-X^d+(X+1)^d$ has degree $d-1$ and has no repeated roots. Notice also that none of its roots is $0$.
Since $\fchar(K)$ does not divide $\ell_0$, the polynomial
$f_0(x)=F_0\left(x^{\ell_0}\right)$ also has no repeated roots; clearly,
$$\begin{aligned}\deg(f_0)=\ell_0 \deg(F)=\ell_0 (d-1)=\ell_0 d-d\\=m_0-\ell_0=(n-m_0+\ell_0)+m_0-\ell_0=n.\end{aligned}$$
Applying Proposition \ref{orderM0} to
$$a=0, B=1, \quad q(x)=1,  v(x)=x^{\ell_0}+1, \quad f(x)=f_0(x),$$
we conclude that $\mathcal Q_0=(0,1)$ is a torsion $K_0$-point of order $m_0$ on $\CC_{f_0,d}$.
This proves (ii) and (i).

Let us prove (iii). 
In light of  Theorem \ref{m0Norm}, we may assume  that there exists a nonzero $B \in K_0$  such that
$$f(x)=-B^d x^{m_0}+(B x^{\ell_0}+1)^d=-(B x^{\ell_0})^d+(B x^{\ell_0}+1)^d$$
and $\mathcal P=(0,1) \in \CC_{f,d}(K_0)$.
Pick any $B_0 =\sqrt[\ell_0]{B} \in K$. Then
$$f(x)=-\left((B_0 x)^{\ell_0}\right)^ d+\left((B_0 x)^{\ell_0}+1\right)^d=f_0(B_0 x).$$
It remains to notice that the map
$$(x,y)\mapsto (B_0 x,y)$$
defines the isomorphism of $K$-curves $\CC_{f_0,d}$ and $\CC_{f,d}$ which sends
$\mathcal O_{f_0,d}$ to $\mathcal O_{f,d}$ and $\mathcal  P$ to $\mathcal Q_0$.
\end{proof}

Now let us discuss the case $n-m_0+\ell_0=1$.

\begin{thm}
\label{equal1}
Suppose that $n-m_0+\ell_0=1$. Suppose that $\fchar(K_0)$ does not divide $\ell_0$.
Let $f(x) \in K_0[x]$
be a degree $n$ polynomial without repeated roots.
Then the following conditions are equivalent.
\begin{itemize}
\item[(i)]
The curve $\CC_{f,d}$ has a $K_0$-point $\mathcal P$ of order $m_0$.
\item[(ii)]
There exist nonzero $B, B_1 \in K_0$ that enjoy the following properties.
\begin{itemize}
\item[(a)]
The polynomial
$$f_{B,B_1}(x)=-B^d x^{m_0}+(B x^{\ell_0}+B_1 x+1)^d$$
has no repeated roots.
\item[(b)]
There is an isomorphism of $K_0$-curves
$\Psi: \CC_{f,d} \to \CC_{f_{B,B_1},d}$ that sends $\mathcal O_{f,d}$ to $\mathcal O_{f_{B,B_1},d}$
and $\mathcal P$ to $(0,1) \in \CC_{f_{B,B_1},d}(K_0)$.
\end{itemize}
\end{itemize}
If these equivalent conditions hold, then $(0,1)$ is a $K_0$-point of order $m_0$ on
$\CC_{f_{B,B_1},d}$ and $\mathcal  Q_d=(-B_0/B,0)$ is a $K_0$-point of order $d$ on
$\CC_{f_{B,B_1},d}$. In addition, $\Psi^{-1}(\mathcal Q_d)$
is a $K_0$-point of order $d$ on
$\CC_{f,d}$.

\end{thm}

\begin{proof}
In light of Theorem \ref{m0Norm} and taking into account that $n-m+\ell_0=1$, we may assume that
there exist a nonzero $B \in K_0$ and a linear polynomial $q(x)=B_1 x+1$ with nonzero $B_1 \in K_0$ such that
$$f(x)=-B^d x^{m_0}+(B x^{\ell_0}+B_1 x+1)^d=f_{B,B_1}(x),$$
and $\mathcal P =(0,1) \in \CC_{f,d}(K_0)$ is a point of order $m_0$. We have
$$f_{B,B_1}(x)=-B^d \left(x^{\ell_0}\right)^d+(B x^{\ell_0}+B_1 x+1)^d=(B x^{\ell_0}+B_1 x+1)^d-\left(B x^{\ell_0}\right)^d.$$
It follows that $x_0=-1/B_1 \in K_0$ is a root of $f_{B,B_1}(x)$, which means that
$(-1/B_1,0)$ is a $K_0$-point of order $d$ on $\CC_{f_{B,B_1},d}$.
\end{proof}

\section{Elliptic curves with points of order 4}
\label{ellipticCurves}
\begin{thm}\label{thm5}
 \label{elliptic}
 Let $K_0$ be a perfect field of characteristic $\ne 2$. Let $E$ be an elliptic curve over $K_0$. Then the following conditions
 are equivalent.
 \begin{itemize}
 \item[(i)]
 $E(K_0)$ contains a point $\mathcal P$ of order $4$.
 \item[(ii)]
 There exist nonzero $B, B_1 \in K_0$ such that
 $B_1^2-8 B \ne 0$
 and $
 E$ is isomorphic over $K_0$ to the elliptic curve
 $$\mathcal{E}_{B,B_1}:y^2=(2 B x^2+B_1 x+1 )(B_1x+1).$$
 in such a way that under this isomorphism $\mathcal P$ goes to the point
 $\mathcal  Q_0=(0, 1) \in \mathcal{E}_{B_2,B_1}(K_0).$
 \end{itemize}
 If these equivalent conditions hold, then $\mathcal Q_2=(-1/B_1,0)$ is a $K_0$-point of order 2
 on $\mathcal{E}_{B,B_1}$ and
 $$\mathcal Q_2=2\mathcal Q_0.$$
 \end{thm}

 \begin{Rem}
 \label{Discr}
 Recall that $\fchar(K_0) \ne 2$.
 Let $B, B_1$ be nonzero elements of $K_0$. Then the cubic polynomial
 $$f_{B,B_1}(x)=(2 B x^2+B_1 x+1 )(B_1x+1)\in K_0[x]$$
 has no repeated roots if and only if
  $$B_1^2-8 B \ne 0.$$
   Indeed, the discriminant $\Delta$ of the quadratic factor $2 Bx^2+B_1 x+1$ is
$B_1^2-8 B_1$. So  $f_{B,B_1}(x)$ is separable if and only if  $\Delta \ne 0$
   and the root $x_0=-1/B_1$ of the linear factor $B_1 x+1$ is not a root of the quadratic factor;
   the latter condition means that
   $0 \ne 2 B (-1/B_1)^2+ B_1 (-1/B_1)+1$, i.e.,
   $$\frac{2B}{B_1^2} -1+1=\frac{2B}{B_1^2} \ne 0.$$
   But this condition is fulfilled automatically, because $B \ne 0$.
   \end{Rem}

 \begin{proof}
 If we take $n=3$ and  $d=2$, then obviously
 $$m_0=4, \ \ell_0=2.$$

 Except the equality $\mathcal Q_2=2 \mathcal Q_0$, all the assertions of Theorem \ref{elliptic} follow readily
 from Theorem \eqref{equal1} combined with Remark \ref{Discr}.
 In order to check the remaining equality, notice that the tangent line $L$
 to the plane affine curve $y^2-f_{B,B_1}(x)$ at $\mathcal Q_0=(0,1)$ is defined by the equation
 $2 (y-1)-2 B_1 x=0$ and the point $\mathcal Q_2=(-1/B_1,0)$ lies on $L$.  Since $\mathcal Q_0$ has order $4$,
 it is not an inflection point, which implies that
 $2 \mathcal Q_0+\mathcal Q_2=\mathcal O$ on $E$. Since $\mathcal Q_2$ has order $2$, we get the desired $\mathcal Q_2=2\mathcal Q_0$.
 \end{proof}

 \begin{Rem}\label{rem2}
 The paper of Kubert \cite[Table 3 on p. 217]{Kubert} contains a family of elliptic curves
 \begin{equation}
 \label{Eb0}
 E(b,0): y^2+xy-by=x^3-bx^2
 \end{equation}
 with a torsion point $\mathcal P=(0,0)$ of order $4$. Here $b^4(1+16 b) \ne 0$. (See also  \cite[Appendix E, p. 186]{Lozano}.)
 Let us explain how $E(b,0)$  and $\mathcal P$ could be obtained as a specialization of our versal family $\mathcal{E}_{B,B_1}$
 with
 $$B=\frac{-2}{b}, \quad B_1=\frac{-1}{b}.$$
 First, we rewrite the equation \eqref{Eb0} as
 $$\left(y+\frac{x-b}{2}\right)^2=x^3-bx^2+\left(\frac{x-b}{2}\right)^2$$
 or equivalently,
 $$y_1^2= \left(2x^2+\frac{x-b}{2}\right) \frac{x-b}{2}$$
 where $y_1:=y+\frac{x-b}{2}$. Dividing by $(-b/2)^2$, we get the equation
 $$\left(\frac{-2y_1}{b}\right)^2=\left(\frac{-4}{b}x^2+\frac{-1}{b}x+1\right)\left(\frac{-1}{b}x+1\right).$$
 Notice that the RHS is nothing else but the polynomial $f_{-2/b,-1/b}(x)$. If we put $y_2:=-2y_1/b$,
 then we get the equation $y_2^2=f_{-2/b,-1/b}(x)$, i.e., $E(b,0)$ is isomorphic to
 $\mathcal{E}_{-2/b,-1/b}$ in such a way that $\mathcal P=(0,0)$ goes under this isomorphism to the point $(0,1)$. \end{Rem}
\begin{prop}\label{notversal}
The family of elliptic curves $ E(b,0)$ with a point of order 4 is versal.
\end{prop}
\begin{proof} Since
by Theorem \ref{thm5} we know that for any perfect field of characteristic $\ne 2$ the family
$$\mathcal{E}_{B,B_1}:y^2=(2 B x^2+B_1 x+1 )(B_1x+1)$$
of elliptic curves with a point of order 4 is versal, it is sufficient to prove that for all nonzero $B$ and $B_1$ there exists
a nonzero $b$ such that $E(b,0)$ is $K_0$-isomorphic to $\mathcal{E}_{B,B_1}$.
Let us denote $2B=a$ and $B_1=c$.  Then
$$\begin{aligned}(2 B x^2+B_1 x+1 )(B_1x+1)=(a x^2+c x+1)(c x+1)\\=ac x^3+(a+c^2)x^2+2c x+1.\end{aligned}$$
Consider the curve \beq\label{ord4_1}y^2=acx^3+(a+c^2)x^2+2cx+1.\eeq
Multiplying both sides of \eqref{ord4_1} by $4(ac)^2$ and using the change of variables $x_1=acx$, $y_1=2acy$, we get
\beq\label{ord4_3}
y_1^2=4x^3_1+4(a+c^2)x^2_1+8ac^2x_1+4a^2c^2.\eeq
Consider the curve $$E(b,0):y^2+xy-by=x^3-bx^2,$$
which by Remark \ref{rem2} is $K_0$-isomorphic to the curve
$$y^2=\left(\frac{-4}{b}x^2+\frac{-1}{b}x+1\right)\left(\frac{-1}{b}x+1\right),$$
which, in turn, is obviously $K_0$-isomorphic to the  curve
$$y^2=(4x^2+x-b)(x-b),$$
i.e., to the curve
\beq\label{E(b,0)1}y^2=4x^3+(1-4b)x^2-2bx+b^2.\eeq For any  nonzero
$a$ and $c$, $c\neq0$, take $$b=-\frac{a}{4c^2}.$$
Then \eqref{E(b,0)1} takes the form
\beq\label{E(b,0)2}
y^2=4x^3+\left(1+\frac{a}{c^2}\right)x^2+\frac{a}{2c^2}x+\frac{a^2}{16c^4}.\eeq
Multiplying both sides of \eqref{E(b,0)2} by $4^3c^6$, we get
\beq\label{E(b,0)3}
4^3c^6y^2=4\cdot 4^3c^6x^3(a+c^2)c^4\cdot 4^3x^2+8\cdot 4ac^2x+4a^2c^2.\eeq
Now the change of variables $2^3c^3y=y_2$, $4c^3x=x_2$
leads to equation
$$y_2^2=4x^3_2+4(a+c)^2x^2_2+8ac^2x_2+4a^2c^2,$$
which ends the proof.
\end{proof}

 \section{Curves  with  two packets of torsion points of order $m_0$}
\label{2 packets}
\begin{lem}\label{0 and -1}
Let $  r_1$ and $  r_2$ be integers that are strictly greater than $1$,
$$\mathcal P=(a_0,c_0),\ \mathcal Q=(a_1,c_1)\in C_{f,d}(K)\subset\CC_{f,d}(K)$$
 be points of order $  r_1$ and $ r_2$ respectively
with
$$x(\mathcal  P)=a_0\in K, \quad x(\mathcal Q)=a_1 \in K.$$
Assume that  $a_0\neq a_1$.
Then there exist a degree $n$ polynomial $\tilde f(x)\in K[x]$ without repeated roots and  a biregular isomorphism of plane affine curves
 $$\Psi: C_{\tilde{f},d} \to C_{f,d}$$
 that extends to the biregular isomorphism $\bar{\Psi}$ of the cyclic covers
 $$
\begin{CD}
\CC_{\tilde f,d} @>>> \CC_{f,d}\\
@VVV @VVV\\
\mathbb P^1@>>> \mathbb P^1
\end{CD}
$$
sending $\mathcal O_{\tilde f,d}$ to $\mathcal O_{f,d}$ (and so preserving the order of points) and two points with abscissas $0$ and $-1$ on $\CC_{\tilde f,d}$ respectively to $\mathcal P$ and $\mathcal Q$.
\end{lem}
\begin{proof}
  Let us consider the automorphism of the affine line
 $$S: z \mapsto \lambda z+\mu, \quad \mu=a_0, \ \lambda=a_0-a_1\ne 0.$$
 Then
 $$S(0)=a_0, \quad S(-1)=a_1.$$
 and
 $$\tilde{f}(x):= \lambda^{-n}f(Sx)\in K[x]$$
 is a monic degree $n$ polynomial without repeated roots. Choose
 $$\lambda^{n/d}:=\sqrt[d]{\lambda^n}\in K$$ and consider  the smooth plane affine curve
 $$C_{\tilde{f},d}: w^d=\tilde{f}(z)$$
 and the biregular isomorphism of plane affine curves
 $$\Psi: C_{\tilde{f},d} \to C_{f,d}, \quad (z,w) \mapsto (x,y)=\left(S(z), \lambda^{n/d}w\right),$$
 which extends to the biregular isomorphism $\bar{\Psi}:\CC_{\tilde{f},d} \to \CC_{f,d}$ sending
We have
 $$\begin{aligned}
 \tilde{\mathcal P}=&\left(0, \ c_0/\lambda^{n/d}\right),  \ \tilde{\mathcal Q}=\left(-1\ ,c_1/\lambda^{n/d}\right) \in C_{\tilde{f},d}(K)\subset \CC_{\tilde{f},d} (K), \\
&\bar{\Psi}(\tilde{\mathcal P})=\Psi\left((0,  \ c_0/\lambda^{n/d})\right)=(a_0,c_0)=\mathcal P,\\
 &\bar{\Psi}(\tilde{\mathcal Q})=\Psi\left((-1 ,\  c_1/\lambda^{n/d})\right)   =(a_1,c_1)=\mathcal Q.
 \end{aligned}$$
 This implies that  $\tilde{\mathcal P}$ and $\tilde{\mathcal Q}$ are points of order $  r_1$  and $ r_2$ respectively on
 $\CC_{\tilde{f},d}$ with $z$-coordinates $0$ and $-1$ respectfully.
 \end{proof}

Assume that there is a curve with at least 2 packets of points of order $m_0$.
By Lemma \ref{0 and -1}
we may assume that the abscissas of these points are $0$ and $-1$. Therefore  there exist polynomials $u(x)$ and $v(x)$ of degree $\ell_0$ such that
$$f(x)=A_1x^{m_0}-u(x)^d=A_2(x+1)^{m_0}-v(x)^d,$$
where $A_1$ and $A_2$ are nonzero elements of $K$.
The converse statement is also true, namely, if there exist polynomials $u(x)$ and $v(x)$ of degree $\ell_0$ such that
$$f(x)=A_1x^{m_0}-u(x)^d=A_2(x+1)^{m_0}-v(x)^d,$$
where $A_1$ and $A_2$ are nonzero elements of $K$, then the curve $\CC_{\tilde{f},d}$ has (at least 2) packets of order $m_0$ points with abscissas $0$ and $-1$  (see \cite{BZ}).

\begin{example}
 \label{m0nplus1}
 Recall that $(d,\fchar(K)=1$ and $m_0=d \ell_0$.
 Suppose  that $m_0=n+1$ and $\fchar(K) \nmid m_0$, i.e.,   $\fchar(K) \nmid \ell_0$. Let us put
 $$f(x):=(x+1)^{m_0}-x^{m_0}=m_0 x^n + \text{terms of lower degree}.$$
 It is well known (see, for example \cite{BZ}) that $f(x)$ has no repeated roots. On the other hand, if we put
 $$A_2:=1, \ v(x):=x^{\ell_0}; \quad A_:=-1, \gamma:=\sqrt[d]{-1}, \ u(x):=\gamma (x+1)^{\ell_0}$$
 then
 $$f(x)=A_1 x^{m_0}-u(x)^d= A_2 (x+1)^{m_0}-v(x)^d.$$
 This implies that all the points on $\CC_{f,d}$ with abscissas $0$ and $-1$ are torsion points of order $m_0$.
 \end{example}

\begin{thm}
Suppose that
$$\fchar(K)=0,  \ n-m_0+\ell_0\ge 0, \ m_0\neq n+1,$$
 and $f(x)\in K[x]$ is a degree $n$ polynomial without repeated roots.

Suppose that $\CC_{f,d}$ has at least 2 packets of torsion points of order $m_0$. Then:
 \begin{itemize}
 \item[(i)]
 $d\leq 5$.
 \item[(ii)]
 If $n<2d$, then
 $$d \le 4,  \quad m_0=2d, \quad n \le 7.$$
 \item[(iii)]
 If $n=7$, then $d=3$.
 \item[(iv)]
 $n \ne 5,6$.
  \item[(v)]
  If $n=4$, then $d=3$.
   \item[(vi)]
   If $n=9$, then $d=4$.
 \end{itemize}
  \end{thm}
 \begin{proof}
Since $\CC_{f,d}$ has 2 packets of points of order $m_0$,
 there exist polynomials $u(x)$ and $v(x)$ of degree $\ell_0$ such that
\beq\label{2 pairs}f(x)=A_1x^{m_0}-u(x)^d=A_2(x+1)^{m_0}-v(x)^d,\eeq
where $A_1$ and $A_2$ are nonzero elements of $K$.
The polynomial $u(x)$ has the form  $u(x)=Bx^{\ell_0}+q(x)$, where $B\neq0$ and the polynomial  $q(x)$ has degree $n-m_0+{\ell}_0$.
In the field of rational functions $K(x)$ we have
\beq\label{2packets gen1}
A_1-\left(\frac{u(x)}{x^{\ell_0}}\right)^d=A_2\left(1+\frac1x\right)^{m_0}-\left(\frac{v(x)}{x^{\ell_0}}\right)^d,
\eeq
\beq\label{2packets gen2}
A_1=A_2(1+t)^{m_0}+\tilde u(t)^d+\tilde v(t)^d,
\eeq
where $\tilde u(t)= t^{\ell_0}u(1/t)$, $\tilde v(t)=\eta t^{\ell_0}v(1/t)$ are degree $\ell_0$ polynomials in $t$ and $\eta\in K$,  $\eta^d=-1$.
Consequently, the equation
\beq\label{poleq1}
X^d+Y^d+Z^d=A_1,
\eeq
where $A_1$ is a nonzero constant, has a solution in polynomials $f_1,f_2,f_3$ of degree ${\ell}_0\geq 2$,
namely, $f_1=\sqrt[d]{A_2}(1+t)^{\ell_0}$, $f_2= \tilde u(t)$,
$f_3=\tilde v(t)$.

First let us prove that the polynomials  $f_1^d,f_2^d,f_3^d$ are linearly independent over $K$. Assume the contrary.
 Let $\alpha f_1^d+\beta f_2^d+\gamma f_3^d=0$, where at least one of the coefficients $\alpha,\beta,\gamma$ is nonzero. If one of the coefficients is zero, then obviously the other two coefficients are nonzero. Assume that $\alpha=0$. Then  $\beta\neq0$, $\gamma\neq0$ and $f_2^d=D f_3^d$, where $D= -\gamma/\beta\neq0$.
Consequently,
  $$\tilde u(t)^d=D\tilde v(t)^d, \quad
t^mu(1/t)^d=-Dt^mv(1/t)^d.$$
Then $u(x)^d=-Dv(x)^d$, and from \eqref{2 pairs} we obtain
$$A_1x^m-D v(x)^d=A_2(x+1)^m-v(x)^d.$$
It follows that
\beq\label{multroots}A_1x^{m_0}-A_2(x+1)^{m_0}=(D-1) v(x)^d.\eeq
If $D=1$ we have a contradiction. If $D\neq 1$, then we again get a contradiction since the left-hand side of \eqref{multroots}
does not have multiple roots.

Now let $\beta=0$. Then $f_1^d=D f_3^d$, where $D=-\gamma/\alpha\neq0$. It follows that
$$A_2(1+t)^{m_0}=-D\tilde v(t)^d.$$
In particular, $-1$ is the only root of $\tilde v(t)$.
However, the roots of the polynomials $v(x)$ and $\tilde v(x)$ are reciprocal  which implies that $-1$
is the only of $v(x)$. But this is impossible since in this case the polynomial
$f(x)=A_2(x+1)^{m_0}-v(x)^d$ would have a multiple root $-1$.

 Now let $\gamma=0$. Then by the same reasoning as in the case $\beta=0$  we obtain that $-1$ is the only root of $\tilde u(t)$ and consequently
 the only root of $u(x)$.
 Thus $u(x)=\sqrt[d]{A_1}(x+1)^{\ell_0}$, which contradicts the fact that
  the  polynomial $u(x)$ has the form  $u(x)=Bx^{\ell_0}+q(x)$, where $B\neq0$ and $q(x)$ has degree $n-m_0+{\ell_0}$.
Indeed, if  $u(x)=\sqrt[d]{A_1}(x+1)^{m_0}$, then $n-m_0+\ell_0=\ell_0-1$ and $m_0=n+1$,  a contradiction.

Thus, all the coefficients $\alpha,\beta,\gamma$ are distinct from zero and
 $$(\sqrt[d]{\alpha} f_1)^d+(\sqrt[d]{\beta}f_2)^d+(\sqrt[d]{\gamma}f_3)^d=0,$$
 which is impossible if $d\geq 3$. Now let us consider the Wronskian
\beq\label{wronskian}W(f_1^d,f_2^d,f_3^d) =\begin{vmatrix}
f_1^d&f_2^d&f_3^d\\
(f_1^d)'&(f_2^d)'&(f_3^d)'\\
(f_1^d)''&(f_2^d)''&(f_3^d)''\\
\end{vmatrix}
\eeq
of the polynomials $f_1^d,f_2^d,f_3^d$, which is nonzero since  the polynomials $f_1^d,f_2^d,f_3^d$ are linearly independent. Since $f_3^d=A_2-f_1^d-f_2^d$, we obtain that
\beq\label{poleq2}\begin{aligned}W(f_1^d,f_2^d,f_3^d)&=W(f_1^d,f_2^d,A)=\begin{vmatrix}
f_1^d&f_2^d&A_2\\
(f_1^d)'&(f_2^d)'&0\\
(f_1^d)''&(f_2^d)''&0\\
\end{vmatrix}\\
&=A_2((f_1^d)'(f_2^d)''-(f_2^d)'(f_1^d)'').
\end{aligned}
\eeq
Now let us estimate the degrees of  $W(f_1^d,f_2^d,f_3^d)$ and $A_2((f_1^d)'(f_2^d)''-(f_2^d)'(f_1^d)'')$.
Since
 $(f_i^d)'=df_i^{d-1}f_i'$ and $$(f_i^d)''=d(d-1)f_i^{d-2}(f_i')^2+df_i^{d-1}f_i''=f_i^{d-2}(d(d-1)(f_i')^2+
 df_{i}f_i''),$$
 we obtain that all elements in the $i$th column in \eqref{wronskian} are divisible by $f_i^{d-2}$ and so
 \beq\label{wronskian1}W(f_1^d,f_2^d,f_3^d) =f_1^{d-2}f_2^{d-2}f_3^{d-2}P,
 \eeq
 where $ P$ is a nonzero polynomial. If $f_1^d,f_2^d,f_3^d$ are linearly independent over $K$, then $ W(f_1^d,f_2^d,f_3^d)$ is a nonzero polynomial, and so $P$ is a nonzero polynomial.  Therefore, \beq \label{degwronskian}\deg W(f_1^d,f_2^d,f_3^d)\geq \deg f_1^{d-2}f_2^{d-2}f_3^{d-2}=3\ell_0(d-2).\eeq
On the other hand,  the degree of $A_2((f_1^d)'(f_2^d)''-(f_2^d)'(f_1^d)'')$ does not exceed
$(\ell_0 d-1)+(\ell_0 d-2)=2l_0d-3$, and so $\deg W(f_1^d,f_2^d,f_3^d)\leq 2\ell_0 d-3$.
 It follows that
 $3\ell_0(d-2)\leq 2\ell_0 d-3$, i.e., $\ell_0 d-6\ell_0 \le -3$, i.e.,
 \begin{equation}
 \label{ellD6}
 \ell_0 (d-6) \le -3.
 \end{equation}
 This implies that $d<6$, i.e., $d\leq 5$, which proves (i). In addition, it follows from \eqref{ellD6} that if
 $\ell_0=2$ (i.e., $n<2d$), then $d \ne 5$, i.e., $2 \le d \le 4$.
 In this case
 $$n<m_0=\ell_0 d= 2d \le 2 \cdot 4=8;$$
 in particular, $n<8$, i.e., $n \le 7$. This proves (ii).

 In order to prove (iii), first let us assume that $n=7$. Since $8=n+1$ is {\sl not} divisible by $d$,
 we conclude that $d \ne 2,4$.

 Suppose that $d \ne 3$.
 Then $d=5$. It follows that
 $$m_0=10, \ell_0=2,  \ n-m_0+\ell_0=7-10+2=-1,$$
 which is not possible, because if $n-m_0+\ell_0<0$, then there are no points of order $m_0$.
 The obtained contradiction implies that $d=3$, which finishes the proof  of (iii).

  If $n=5$, then $d$ is neither $5$ nor a divisor of $n+1=6$. Hence, $d \ne 2,3$. This implies
 that $d=4$ and therefore
 $$m_0=8, \ell_0=2, \  n-m_0+\ell_0=5-8+2=-1<0.$$
  Then there are no points of order $m_0$.
 The obtained contradiction implies that $n \ne 5$, which proves first part of (iv).

 If $n=6$, then $(6,d)=1$, i.e., $d \ne 2,3,4$. Hence $d=5$.
 This implies that
 $$m_0=10, \ell_0=2, \ n-m_0+\ell_0=6-10+2=-2<0.$$
 Then there are no points of order $m_0$.
 The obtained contradiction implies that $n \ne 6$, which ends the proof of (iv).

 If $n=4$, then $d$ is not a divisor of $4$; in addition, $d<4$. The only remaining possibility is
 $d=3$, which proves (v).

 If $n=9$, then $d \ne 3$ and $(d,n+1)=(d,10)=1$; the latter means that $d \ne 2,5$.
 It follows that $d=4$, which proves (vi).
 \end{proof}
\section {Hyperelliptic curves}
\label{d=2}
We keep the notation and assumptions of the previous section.

\subsection*{The case $d=2$.}
Since $(n,2)=(n,d)=1$,  the degree $n$ of $f(x)$ is odd. Due to $n<m_0<n+2$ we have $m_0=n+1$.
If $\CC_{f,d}$ has two packets of points of order $m_0$ with abscissas, say, $0$ and $-1$, then $f(x)$ can be represented in the form
\beq\label{2pairs-1}f(x)=A_1x^{n+1}-u(x)^2=A_2(x+1)^{n+1}-v(x)^2,\eeq  where $u(x),v(x)$ are polynomials of degree $\ell_0=(n+1)/2$.
Let  $B_1x^{\ell_0}$ and $B_2x^{\ell_0}$, where $B_1\neq0$ and $B_2\neq0$ be the leading terms of $u(x)$ and $v(x)$ respectively.
Then $B_1^2=A_1$
 and $B_2^2=A_2$. We have
\beq\label{2pairs0}A_1(x^{n+1}-\tilde u(x)^2)=A_2((x+1)^{n+1}-\tilde v(x)^2),\eeq
where $$\tilde u(x)=\frac{1}{B_1}u(x),\; \tilde v(x)=\frac{1}{B_2}v(x).$$ Let $C\in K$ be such that
$ C^{n+1}=A_1/A_2$ if $A_1\neq A_2$ and $C=1$ if $A_1=A_2$. Then
\beq\label{2pairs-2}C^{n+1}(x^{n+1}-\tilde u(x)^2)=(x+1)^{n+1}-\tilde v(x)^2,\eeq
  $$(Cx)^{n+1}-(C^l\tilde u(x))^2=(x+1)^{n+1}-\tilde v(x)^2,$$
\beq\label{2pairs1}(x+1)^{n+1}-(Cx)^{n+1}=\tilde v(x)^2-(C^{\ell_0}\tilde u(x))^2.\eeq
 Let $\mu_{n+1}$ be the set of all $(n+1)$-th roots of $1$. Then
 \beq\label{2pairs2}(x+1)^{n+1}-(Cx)^{n+1}=\prod\limits_{\eps\in \mu_{n+1}}(x+1-C\eps x)=
 \prod\limits_{\eps\in\mu_{n+1}}((1-C\eps)x+1).
 \eeq
Hence
\beq\label{2pairs3}
 \prod\limits_{\eps\in\mu_{n+1}}((1-C\eps)x+1)=(\tilde v(x)+C^l\tilde u(x))(\tilde v(x)-(C^l\tilde u(x)).
 \eeq
 Consider the  cases $A_1\neq A_2$ and $A_1=A_2$ separately.

 a) Let $A_1\neq A_2$. Since the degree of each factor on the right-hand side of \eqref{2pairs3} is at most $\ell_0$ and the degree of the polynomial on the left-hand side of \eqref{2pairs3} is $2\ell_0$, we have
 $\deg(\tilde v(x)+C^{\ell_0}\tilde u(x))=\ell_0$ and  $\deg(\tilde v(x)-C^{\ell_0}\tilde u(x))=\ell_0$.
For any $\ell_0$-element subset $I$ of $\mu_{n+1}$ we put
\beq\label{H_I}H_I= \prod\limits_{\eps\in I}((1-C\eps)x+1).\eeq
 Since the polynomial on the left-hand side of \eqref{2pairs3}  has no repeated roots and the factors on the right-hand side of \eqref{2pairs3} have degree $\ell_0$, they are relatively prime. It follows that
 \beq\label{2pairs4}
\tilde v(x)+C^{\ell_0}\tilde u(x)=\lambda H_I \;\text{and}\;\tilde v(x)-C^{\ell_0}\tilde u(x)=\frac1{\lambda}H_{\complement I}
 \eeq
 for some $\ell_0$-element subset $I$ of $\mu_{n+1}$ and $\lambda\in K$, $\lambda\neq0$. We have
 \beq\label{2pairs5}
\tilde v(x)=\frac{\lambda H_I+(1/{\lambda})H_{\complement I}}2, \;\tilde u(x)=\frac{\lambda H_I-(1/{\lambda})H_{\complement I}}{2C^{\ell_0}}.\eeq It will follow from Theorem \ref{finite} that the curve
 $\CC_{f,2}$, where the polynomial $f(x)$ is defined by \eqref{2pairs-1} has at least 2 packets of torsion points of order $m_0$.

 b) Let $A_1=A_2$. Then the equation \eqref{2pairs-2} takes the form
 \beq\label{2pairs-3} x^{n+1}-\tilde u(x)^2 =(x+1)^{n+1}-\tilde v(x)^2,\eeq
 hence
 \beq\label{2pairs-4}(x+1)^{n+1}-x^{n+1}=\tilde v(x)^2-\tilde u(x)^2.\eeq
 The left-hand side of \eqref{2pairs-4} is a polynomial of degree $n$  with leading coefficient $n+1$.
The right-hand side of \eqref{2pairs-4} is the product
\beq\label{2pairs{-5}} (\tilde v(x)-\tilde u(x)(\tilde v(x)+\tilde u(x)).\eeq
The degree of each product is at most $\ell_0$ and cannot be both less that $\ell_0$ or equal to $\ell_0$ simultaneously.
Thus one of them has degree $\ell_0=(n+1)/2$ and the other $\ell_0-1=(n-1)/2$.
We have
\beq\label{C=1_1}
(x+1)^{n+1}-x^{n+1}= (\tilde v(x)-\tilde u(x)(\tilde v(x)+\tilde u(x)).
\eeq
The polynomial $(x+1)^{n+1}-x^{n+1}$ has $n$ distinct roots $(\eps-1)^{-1}$, where
$\eps\in\mu_{n+1}$.
For any $\ell_0$-element subset $I$ of $\mu_{n+1}$ we put
\beq\label{C=1_2}H_I= \prod\limits_{\eps\in I}((1-\eps)x+1) \quad \text{and} \quad H_{\complement I}= \prod\limits_{\eps\in \complement I}((1-\eps)x+1).\eeq
We have \beq\label{C1_3} H_IH_{\complement I}=(\tilde v(x)-\tilde u(x))(\tilde v(x)+\tilde u(x)).\eeq
There exist an $\ell_0$-element subset of $\mu_{n+1}$ and a nonzero $\lambda\in K$ such that
\beq\label{C=1_3} \tilde v(x)-\tilde u(x)=\lambda H_I, \; \tilde v(x)+\tilde u(x)=\frac1{\lambda}H_{\complement I}\eeq
or
\beq\label{C=1_4}\tilde v(x)-\tilde u(x)=\frac1{\lambda}H_{\complement I}, \; \tilde v(x)+\tilde u(x)=\lambda H_{I}\eeq
From \eqref{C=1_3} we get
\beq\label{C=1_5}\tilde v(x)=\frac{\lambda H_I+(1/\lambda)H_{\complement I}}2,
\;\tilde u(x)=\frac{(1/\lambda)H_{\complement I}-\lambda H_I}2.\eeq
From \eqref{C=1_4} we get
\beq\label{C=1_6}\tilde v(x)=\frac{\lambda H_I+(1/\lambda)H_{\complement I}}2,
\;\tilde u(x)=\frac{\lambda H_I-(1/\lambda)H_{\complement I}}2.
\eeq
Let \eqref{C=1_5} be the case (the case \eqref{C=1_6} is considered similarly).
Then  for any $\ell_0$-element $I\subset \mu_{n+1}$ and $\lambda\in K$, $\lambda\neq0$, the polynomials $\tilde v(x)$ and $\tilde u(x)$ defined by
 \eqref{C=1_5} satisfy equation \eqref{C=1_1}. It will follow from Theorem \ref{finite} that the curve
 $\CC_{f,2}$, where the polynomial $f(x)$ is defined by \eqref{2pairs-1} has at least 2 packets of torsion points of order $m_0$.

\begin{thm}\label{finite}
For each $\ell_0$-element subset $I\subset \mu_{n+1}$ there are only finitely many values of   $\lambda\in K^{*}$ such that
the polynomial
$$x^{n+1}-\tilde u(x)^2, \; \text{where}\; \tilde u(x)=\frac{\lambda H_I-(1/{\lambda})H_{\complement I}}{2C^{\ell_0}}$$
has  a multiple root.
\end{thm}
\begin{proof}
1)
We have
$$x^{n+1}-\frac{(\lambda H_I-(1/{\lambda})H_{\complement I})^2}{4C^{n+1}}=\frac{(Cx)^{n+1}-(\lambda H_I-(1/{\lambda})H_{\complement I})^2}{4C^{n+1}}.$$
Hence
\beq\label{mult4}\begin{aligned}(Cx)^{n+1}-(\lambda H_I-(1/{\lambda})H_{\complement I})^2=((Cx)^{\ell_0}-\vf(x))((Cx)^{\ell_0}+\vf(x)),\end{aligned}\eeq
where
$$\vf(x)=\lambda H_I(x)-(1/{\lambda})H_{\complement I}(x).$$
Assume that the polynomials $(Cx)^{\ell_0}-\vf(x)$ and $(Cx)^{\ell_0}+\vf(x)$ have a common root $x_0$.
Then  $(Cx_0)^{\ell_0}=0$ and $\vf(x_0)=0$. Hence $x_0=0$ and
$\lambda H_I(0)-(1/{\lambda})H_{\complement I}(0)=0$. Since $ H_I(0)=H_{\complement I}(0)=1$, we
obtain that
$0=\vf(0)=\lambda-1/{\lambda}$, and so $\lambda=\pm1$. So for all $\lambda\neq\pm1$ the polynomials $(Cx)^{\ell_0}-\vf(x)$ and $(Cx)^{\ell_0}+\vf(x)$ do not have common roots.
To prove the theorem it suffices to prove that for the polynomial $(Cx)^{\ell_0}-\vf(x)$ there are only finitely many values of $\lambda$ for which it has  multiple roots (since the other factor on the right-hand side of \eqref{mult4} is obtained by replacing $\lambda$ by $-\lambda$).

Assume that $x_0$ is a multiple root of
\beq\label{mult5}(Cx)^{\ell_0}-\vf(x)= C^{\ell_0}x^{\ell_0}-(\lambda H_I(x)-(1/{\lambda})H_{\complement I}(x)).\eeq Then
\beq\label{mult1}
C^{\ell_0}x_0 ^{\ell_0}-(\lambda H_I(x_0)-(1/{\lambda})H_{\complement I}(x_0))=0,\eeq
\beq\label{mult1bis}\ell_0 C^{\ell_0}x_0^{\ell_0-1}-(\lambda H_I'(x_0)-(1/{\lambda})H_{\complement I}'(x_0))=0.
 \eeq
If
$$\lambda H_I(x_0)-(1/{\lambda})H_{\complement I}(x_0))=0,$$ then it follows from the first equation of \eqref{mult1}
that $x_0=0$ and
\beq\label{lamex}
\lambda^2=\frac{H_{\complement I}(0)}{H_I(0)} .\eeq
In what follows we assume that  $\lambda\neq\pm\sqrt{H_{\complement I}(0)/H_I(0)}$.
Multiplying   equation \eqref{mult1} by $\ell_0$ and equation \eqref{mult1bis} by $x_0$ we get
\beq\label{mult2}
\begin{aligned}
\ell_0 C^{\ell_0}x_0 ^{\ell_0}-\ell_0(\lambda H_I(x_0)-(1/{\lambda})H_{\complement I}(x_0))=0,\\
\ell_0 C^{\ell_0}x_0^{\ell_0}-x_0(\lambda H_I'(x_0)-(1/{\lambda})H_{\complement I}'(x_0))=0.
\end{aligned}
\eeq
Consequently,
\beq\label{mult3}
\begin{aligned}
\ell_0(\lambda H_I(x_0)-(1/{\lambda})H_{\complement I}(x_0))=x_0(\lambda H_I'(x_0)-(1/{\lambda})H_{\complement I}'(x_0)),\\
\ell_0 (\lambda^2 H_I(x_0)- H_{\complement I}(x_0))=x_0(\lambda^2 H_I'(x_0)- H_{\complement I}'(x_0)),\\
\lambda^2(\ell_0 H_I(x_0)-x_0 H_I'(x_0))=\ell_0 H_{\complement I}(x_0)-x_0 H_{\complement I}'(x_0).
\end{aligned}
\eeq
\begin{lem}\label{lem1}
The polynomial $\ell_0 H_I(x)-xH'(x)$  is nonzero.\end{lem}
\begin{proof}
Let
\beq\label{nonzero1}
H_I(x)=a_0x^{\ell_0}+a_1x^{\ell_0-1}+a_2x^{\ell_0-2}+\cdots+2a_{\ell_0-2}x^2+a_{\ell_0-1}x+a_{\ell_0}.\eeq
Then
\beq\label{nonzero2}
H_I'(x)=\ell_0 a_0x^{\ell_0}+\ell_0 a_1x^{\ell_0-1}+\ell_0 a_2x^{\ell_0-2}+\cdots+2\ell_0 a_{\ell_0-2}x^2+\ell_0 a_{\ell_0-1}x+\ell_0 a_{\ell_0}.\eeq
Consequently,
\beq\label{nonzero3}\begin{aligned}
&\ell_0 H_I(x)-xH_I'(x)\\=&a_1x^{\ell_0-1}+2a_2x^{\ell_0-2}+\cdots+(\ell_0-2)a_{\ell_0-2}x^2+(\ell_0-1)a_{\ell_0-1}x +\ell_0 a_{\ell_0}. \end{aligned}\eeq
If $lH_I(x)-xH'(x)$ is a zero polynomial, then
$a_1=a_2=\cdots=a_{\ell_0}=0$ and  $H_I(x)=a_0x^{\ell_0}$, a contradiction.
\end{proof}
By Lemma\label{lem1} $\ell_0 H_I(x)-xH_I'(x)$ is a nonzero polynomial, which implies that the number of its roots is finite and so may and will assume that $x_0$ is not one of them, i.e., $$\ell_0 H_I(x_0)-x_0H_I'(x_0)\neq0. $$Then
\beq\label{mult7}\lambda^2=\frac{\ell_0 H_{\complement I}(x_0)-x_0 H_{\complement I}'(x_0)}{\ell_0 H_I(x_0)-x_0 H_I'(x_0)}.\eeq
From the first equation of \eqref{mult2} we get
\beq\label{mult6}\begin{aligned}\lambda\ell_0 C^{\ell_0}x_0 ^{\ell_0}=\ell_0(\lambda^2 H_I(x_0)-H_{\complement I}(x_0)),\\
\ell_0 H_I(x_0)\lambda^2-\ell_0 C^{\ell_0}x_0 ^{\ell_0}\lambda-\ell_0 H_{\complement I}(x_0)=0.
\end{aligned}\eeq
 Since $\lambda\neq\pm\sqrt{H_{\complement I}(0)/H_I(0)}$, it follows  that at least one of the coefficients
 of the polynomial on the left-hand side of the second equation in \eqref{mult6} is nonzero, and so
 for each $x_0$ there exist at most 2 values of $\lambda$ satisfying \eqref{mult6}.
We prove the theorem if we prove that it follows from the above conditions that $x_0$ is a root of a nonzero polynomial
whose coefficients do not depend on $\lambda$.
Squaring both sides of equation \eqref{mult6} and applying \eqref{mult7} we obtain
\beq\label{mult8}
\begin{aligned}
&\frac{\ell_0 H_{\complement I}(x_0)-x_0 H_{\complement I}'(x_0)}{\ell_0 H_I(x_0)-x_0 H_I'(x_0)}\ell_0^2C^{n+1}x_0^{n+1}\\
=&\ell_0^2\left(\frac{ \ell_0 H_{\complement I}(x_0)-x_0 H_{\complement I}'(x_0)}{\ell_0 H_I(x_0)-x_0 H_I'(x_0)}H_I(x_0)-H_{\complement I}(x_0)\right)^2,\\
&(\ell_0 H_{\complement I}(x_0)-x_0 H_{\complement I}'(x_0))(\ell_0 H_I(x_0)-x_0 H_I'(x_0))\ell_0^2C^{n+1}x_0^{n+1}\\
=&\ell_0^2( \ell_0 H_{\complement I}(x_0)-x_0 H_{\complement I}'(x_0)H_I(x_0)-(\ell_0 H_I(x_0)-x_0 H_I'(x_0))H_{\complement I}(x_0))^2\\
=&\ell_0^2x_0(H_I'(x_0)H_{\complement I}(x_0)-H_{\complement I}'(x_0)H_I(x_0))^2.
\end{aligned}
\eeq
It follows that $x_0$ is a root of the polynomial
\beq\label{mult9}
\begin{aligned}
(\ell_0 H_{\complement I}(x)-x H_{\complement I}'(x))(\ell_0 H_I(x)-x H_I'(x))\ell_0^2C^{n+1}x^{n+1}\\
-\ell_0^2x(H_I'(x)H_{\complement I}(x)-H_{\complement I}'(x)H_I(x))^2
\end{aligned}
\eeq
and since $x_0\neq0$ it is also a root of
\beq\label{mult10}
\begin{aligned}
(\ell_0 H_{\complement I}(x)-x H_{\complement I}'(x))(\ell_0 H_I(x)-x H_I'(x))\ell_0^2C^{n+1}x^{n}\\
-\ell_0^2(H_I'(x)H_{\complement I}(x)-H_{\complement I}'(x)H_I(x))^2
\end{aligned}
\eeq
There will be finitely many possibilities for $x_0$ if we prove that  polynomial \eqref{mult10} is nonzero.
Assume the contrary. Then we have the equality of polynomials
\beq\label{mult11}\begin{aligned}
(\ell_0 H_{\complement I}(x)-x H_{\complement I}'(x))(\ell_0 H_I(x)-x H_I'(x))\ell_0^2C^{n+1}x^{n}\\=
\ell_0^2(H_I'(x)H_{\complement I}(x)-H_{\complement I}'(x)H_I(x))^2.\end{aligned}\eeq
The polynomial on the left-hand side of  \eqref{mult11} is divisible by $x^n$ and not divisible by $x^{n+1}$ since
$$\ell_0 H_{\complement I}(0)-0\cdot H_{\complement I}'(0)=\ell_0\neq0$$
and
$$\ell_0 H_I(0)-0\cdot H_I'(0)=\ell_0\neq0.$$
However, the largest power of $x$ that divides $$(H_I'(x)H_{\complement I}(x)-H_{\complement I}'(x)H_I(x))^2$$
is even, but $n=2\ell_0-1$, a contradiction.
\vskip.3cm
\end{proof}

\end{document}